\documentclass[12pt]{amsart}
\usepackage{amscd, amssymb,latexsym,amsmath, amscd, amsmath}
\usepackage[all]{xy}
\usepackage{anysize}
\usepackage[sc]{mathpazo} 
\usepackage{color}
\marginsize{2.5cm}{2.5cm}{2.5cm}{2.5cm}
\newtheorem{theorem}{Theorem}[section] 
\newtheorem{corollary}[theorem]{Corollary}
\newtheorem{proposition}[theorem]{Proposition}
\theoremstyle{definition}

\theoremstyle{remark}
\newtheorem{remark}{Remark}

\begin{document}

\title[Toric periods for a $p$-adic quaternion algebra]{Toric periods for a $p$-adic quaternion algebra}
\author{U. K. Anandavardhanan and Basudev Pattanayak}

\address{Department of Mathematics, Indian Institute of Technology Bombay, Mumbai - 400076, India.}
\email{anand@math.iitb.ac.in}

\address{Department of Mathematics, Indian Institute of Technology Bombay, Mumbai - 400076, India.}
\email{basudev@math.iitb.ac.in}

\subjclass{Primary 20C15; Secondary 20C20, 20C33}

\date{}

\begin{abstract}
Let $G$ be a compact group with two given subgroups $H$ and $K$. Let $\pi$ be an irreducible representation of $G$ such that its space of $H$-invariant vectors as well as the space of $K$-invariant vectors are both one dimensional. Let $v_H$ (resp. $v_K$) denote an $H$-invariant (resp. $K$-invariant) vector of unit norm in a given $G$-invariant inner product $\langle ~,~ \rangle_\pi$ on $\pi$. We are interested in calculating the correlation coefficient 
\[c(\pi;H,K) = |\langle v_H,v_K \rangle_\pi|^2.\]
In this paper, we compute the correlation coefficient of an irreducible representation of the multiplicative group of the $p$-adic quaternion algebra with respect to any two tori. In particular, if $\pi$ is such an irreducible representation of odd minimal conductor with non-trivial invariant vectors for two tori $H$ and $K$, then its root number $\varepsilon(\pi)$ is $\pm 1$ and $c(\pi; H, K)$ is non-vanishing precisely when $\varepsilon(\pi) = 1$. 
\end{abstract}

\maketitle

\section{Introduction}\label{introduction}\label{intro}

Let $G$ be a compact group with two given subgroups $H$ and $K$ such that any irreducible representation $\pi$ of $G$ admits at most a one dimensional space of invariant vectors for $H$ as well as $K$. Let $v_H$ (resp. $v_K$) denote an $H$-invariant (resp. $K$-invariant) vector of unit norm in a fixed $G$-invariant inner product $\langle ~,~ \rangle_\pi$ on $\pi$. Then the correlation coefficient of $H$ and $K$ in $\pi$ is defined as
\[c(\pi;H,K) = |\langle v_H,v_K \rangle_\pi|^2.\] 
Note that $c(\pi;H,K)$ is well-defined, since $\langle~,~\rangle_\pi$ is unique up to scalar multiplication, and since $v_H$ and $v_K$ are unique up to a complex number of absolute value $1$. The subgroups $H$ and $K$ are said to be correlated with respect to $\pi$ if $c(\pi;H,K) \neq 0$. The correlation coefficient appears in a variety of contexts and we refer to \cite[\S 8, \S 9]{gro91} and \cite{rob99} for more background material.

In this paper, we compute the correlation coefficient for an irreducible representation $\pi$ of $G=D^\times/F^\times$, where $D$ is the quaternion algebra over a $p$-adic field $F$ with $p$ odd, with respect to any two tori as the given subgroups $H$ and $K$. The restriction of an irreducible smooth representation of $D^\times$ (or ${\rm GL}(2,F)$) to its tori is pioneered by the work of Tunnell \cite{tun83,hsy93,sai93}. Tunnell's results, which are for $(SO(3),SO(2))$, and the work of Prasad for $(SO(4),SO(3))$ \cite{pra90,pra92} are the base cases of more general branching laws and the related Gan-Gross-Prasad conjectures \cite{ggp12}. The results of this paper on the correlation coefficient are in the basic context of Tunnell's results and the striking simplicity of these results (see Theorem \ref{thm-main} and Theorem \ref{thm-main-1}) makes one hope for analogous results in the realm of higher rank branching laws.
   
For an irreducible representation $\pi$ of $D^\times$ to have a fixed vector for any torus, it is necessary that its central character $\omega_\pi$ is trivial. It follows that $\pi$ is self-dual and that its epsilon factor $\varepsilon(1/2,\pi,\Psi)$ is independent of the choice of the additive character $\Psi$ of $F$. This is the local root number which we now denote by $\varepsilon(\pi)$.  By our assumptions, we always have $\varepsilon(\pi) \in \{\pm 1\}$. These local constants of arithmetic significance and their sign are of fundamental importance in the Langlands program and we highlight at this point that one of the key findings of this paper is that the vanishing of the correlation coefficient in some of the situations considered in this paper is determined precisely by the sign of the local root number. We state these as our first results.

\begin{theorem}\label{thm-main}
Let $\pi$ be an irreducible representation of $D^\times$ of odd minimal conductor. Let $S$ and $T$ denote the two non-isomorphic tori for which $\pi$ admits non-trivial fixed vectors. Then the local root number $\varepsilon(\pi)$ is $\pm 1$, and 
\[c(\pi; S, T) \neq 0 \iff \varepsilon(\pi) = 1.\]
Moreover, when $\varepsilon(\pi) = 1$,
\[c(\pi; S, T) = \frac{2}{\dim \pi}.\]
\end{theorem}

We make a remark on the assumption of oddness in Theorem \ref{thm-main}. Under this assumption, the two non-isomorphic tori $S$ and  $T$ admitting fixed vectors for $\pi$ have the property that the normalizer $N(S)$ of $S$ in $D^\times$ intersects $T$ in a non-central element (and vice-versa), and $(S,T)$ is a compatible pair in the sense of Remark \ref{compatible} below. If $\pi$ is of even minimal conductor then the tori admitting fixed vectors need not be compatible and this has to be part of the assumption in order to have an analogous result. 

\begin{theorem}\label{thm-main-1}
Let $\pi$ be an irreducible representation of $D^\times$ of even minimal conductor $f(\pi)$. Let $S$ and $T$ denote two distinct tori for which $\pi$ admits non-trivial fixed vectors. Assume that $N(S)$ intersects $T$ in a non-central element. Then the local root number $\varepsilon(\pi)$ is $\pm 1$, and 
\[c(\pi; S, T) \neq 0 \iff (-1)^{f(\pi)/2}\varepsilon(\pi) = 1.\]
Moreover, when $\varepsilon(\pi) = (-1)^{f(\pi)/2}$,
\[c(\pi; S, T) = \frac{2}{\dim \pi}.\]
\end{theorem}

\begin{remark}\label{dp}
In particular, Theorem \ref{thm-main} (resp. Theorem \ref{thm-main-1}) asserts that the invariant vectors for the two tori are orthogonal precisely when $\varepsilon(\pi) = -1$ (resp. $\varepsilon(\pi) \neq (-1)^{f(\pi)/2}$). One implication follows from \cite[Theorem 4]{pra96} as well. Indeed, by loc. cit., this condition on $\varepsilon(\pi)$ would imply that $\pi$ does not have a fixed vector for $N(S)$. Thus, there exists an element $t \in N(S) \cap T$ that  does not fix the $S$-fixed vector $v_S$. Note that $t \cdot v_S = \alpha v_S$ for some $\alpha \neq 1$. Now,
\[\langle v_S, v_T \rangle = \langle v_S, t^{-1} \cdot v_T \rangle = \langle t \cdot v_S, v_T \rangle = \alpha  \langle v_S, v_T \rangle,\]
and it follows that $c(\pi; S, T) = 0$. 
\end{remark}

The tori in $D^\times$ correspond to quadratic extensions of $F$. By our assumption on the residue characteristic being odd, we have only three distinct such extensions with one of these being unramified. We denote the unramified quadratic extension by $M$ and the ramified quadratic extensions by $K$ and $L$. Let us denote the corresponding tori by the same notations as well. 

The irreducible representations of $D^\times$ are understood by \cite{gg68,ch77,how77b,how77a}, and we refer to \cite{gk80,kz80,bf83} for more details and for the explicit construction of these representations that we make use of in this paper. The essential points are summarized in Section \ref{prelim}. Suffices here to say that they are associated to multiplicative characters of quadratic extensions of $F$ and in this association the representations of even minimal conductor arise out of $M$ and the representations of odd minimal conductor arise out of $K$ or $L$. These constructions are available more generally for any tame division algebra; i.e., when the index of the division algebra and the residue characteristic are relatively prime, and this is the reason why we assume that $p$ is odd. 

Now we state the other main results of this paper. The first one is when $\pi$ has odd minimal conductor and the second one is when $\pi$ has even minimal conductor.

\begin{theorem}\label{thm-odd}
Let $\pi$ be an irreducible representation of $D^\times$ associated to a character $\psi$ of $K^\times$, where $K$ is a ramified quadratic extension of $F$. Assume $\psi$ has trivial restriction to $F^\times$ so that $\psi(\varpi) = \pm 1$. Then,
\begin{enumerate}
\item $c(\pi; M, K) =
\begin{cases}
\frac{2}{\dim \pi}  &\text{if $q \equiv 3 \mod 4$ and $\psi(\varpi) = 1$,} \\
0 &\text{otherwise.} 
\end{cases}$
\item $c(\pi; M, L) = 
\begin{cases}
\frac{2}{\dim \pi}  &\text{if $q \equiv 1 \mod 4$ and $\psi(\varpi) = 1$,} \\
0 &\text{otherwise.} 
\end{cases}$
\end{enumerate}
\end{theorem}
\begin{remark}\label{trivial}
In the situation of Theorem \ref{thm-odd}, the representation $\pi$ always has an $M$-fixed vector and it has a $K$-fixed vector precisely when $q \equiv 3 \mod 4$ and an $L$-fixed vector precisely when $q \equiv 1 \mod 4$. In particular, $c(\pi;K,L) = 0$, $c(\pi; M, K) = 0$ when $q \equiv 1 \mod 4$, and $c(\pi; M, L) = 0$ when $q \equiv 3 \mod 4$ for trivial reasons.
\end{remark}
\begin{remark}
By Remark \ref{trivial}, it follows from Theorem \ref{thm-odd} that two non-trivial invariant vectors - one for $M$ and the other for either $K$ or $L$ - are orthogonal precisely when a certain sign, namely the value of $\psi(\varpi)$, is $-1$. Thus, a sign captures and characterizes the orthogonality of invariant vectors! And, this sign is nothing but the local root number of $\pi$ (cf. Section \ref{proof-main}).
\end{remark}
\begin{remark}\label{compatible}
Note that $(M,K)$ (and $(M,L)$) is a compatible pair in the sense that any other embedding of the unramified extension $M$ and the ramified extension $K$ will give rise to a pair, say $(M^\prime,K^\prime)$, which is simultaneously $G$-conjugate to $(M,K)$; i.e, there exists $g \in D^\times$ such that $M^\prime = g Mg^{-1}$ and $K^\prime = gKg^{-1}$. Thus, the correlation coefficient $c(\pi;M,K)$ (and $c(\pi;M,L)$) is independent of the choice of the embeddings of the field extensions in $D^\times$.
\end{remark}
\begin{theorem}\label{thm-even}
Let $\pi$ be an irreducible representation of $D^\times$ associated to a character $\psi$ of $M^\times$, where $M$ is the unramified quadratic extension of $F$. Assume $\psi$ has trivial restriction to $F^\times$. Let $K=F(\varpi)$ and $L=F(\zeta\varpi)$, where $\varpi$ is a fixed uniformizer of $D$ and $\zeta$ is a fixed primitive $(q^2-1)$-st root of unity in $D^\times$ with $q$ being the residue cardinality of $F$. Then,
\[c(\pi; K, L) = \frac{|1 + \psi(\zeta)|^2}{2 \dim \pi}.\]
\end{theorem}
\begin{remark}\label{trivial1}
In the situation of Theorem \ref{thm-even}, the representation $\pi$ always has fixed vectors for both $K$ and $L$ but not for $M$. Thus, $c(\pi;M,K) = c(\pi;M,L) = 0$, for trivial reasons.
\end{remark}
\begin{remark}\label{rmk-iso}
Observe that $(K,L)$ is not a compatible pair and thus Theorem \ref{thm-even} does depend on the choice of the embeddings of the field extensions. In other words, $c(\pi;K,L)$ depends on the field extensions themselves and is not invariant under $F$-isomorphisms of these extensions. 
\end{remark}
\begin{remark}\label{rmk-iso-1}
Note that Theorem \ref{thm-even} can also be stated for two ramified extensions which are $F$-isomorphic. For instance,
\[c(\pi;F(\varpi),F(\zeta^2\varpi)) =  \frac{|1 + \psi(\zeta^2)|^2}{2 \dim \pi},\]
and
\[c(\pi;F(\varpi),F(\zeta^{\frac{q+1}{2}}\varpi)) = 
\begin{cases}
\frac{2}{\dim \pi} &\text{if $\psi(\zeta^{\frac{q+1}{2}}) = 1$,} \\
0 &\text{if $\psi(\zeta^{\frac{q+1}{2}}) = -1$,} 
\end{cases}
\]
where the two fields are $F$-isomorphic precisely when $q \equiv 3 \mod 4$. The sign $\psi(\zeta^{\frac{q+1}{2}})$ is related to the local root number of $\pi$ and in fact $\psi(\zeta^{\frac{q+1}{2}}) = (-1)^{f(\pi)/2}\varepsilon(\pi)$ (cf. Section \ref{proof-main}).
\end{remark}
\begin{remark}
It is easy to compute the dimension of an irreducible representation $\pi$ of $D^\times$ from its explicit construction and this is given by 
\[\dim \pi = \begin{cases}
2q^{\frac{c-2}{2}} &\text{if $c$ is even,} \\
(q+1)q^{\frac{c-3}{2}} &\text{if $c$ is odd,}
\end{cases}\]
where $c$ is the conductor of $\pi$. 
\end{remark}
We also know how many irreducible representations are there of a given conductor of $D^\times/F^\times$ and this is given by \cite[Proposition 3.5 and Remark on p. 189]{tun78}:
\[\mbox{Number of~} \pi \mbox{~of conductor~} c = \begin{cases}
\frac{q^2-1}{2}q^{\frac{c-4}{2}} &\text{if $c$ is even,} \\
2(q-1)q^{\frac{c-3}{2}} &\text{if $c$ is odd.}
\end{cases}\]
It is an elementary observation to note that \cite[Lemma 2.1]{aj22}
\begin{equation}\label{elem}
\sum_{\pi \in \widehat{G}} \dim \pi \cdot c(\pi;H,K) = \frac{|G||H \cap K|}{|H||K|},
\end{equation}
whenever $G$ is a finite group. It is easy to verify that Theorem \ref{thm-odd} and Theorem \ref{thm-even} are compatible with (\ref{elem}) by considering the finite group $G_c = \frac{D^\times}{F^\times(1+\mathcal P_D^{c-1})}$, and noting that
\[|G_c| = \begin{cases}
2(q+1)q^{3j-1} &\text{if $c=2j+1$,} \\
2(q+1)q^{3j} &\text{if $c=2j+2$},
\end{cases}\]
and 
\[|M_c| = (q+1)q^{\left[\frac{c-2}{2} \right]}, |K_c|=|L_c| = 2q^{(c-2)-\left[\frac{c-2}{2}\right]}, \]
and that all the three subgroups mutually intersect trivially. We refer to Section \ref{schur} for a proof of this verification.

We end the introduction with a few lines about the other results contained in this paper and the proofs of the theorems. Theorem \ref{thm-main} (resp. Theorem \ref{thm-main-1}) follows from Theorem \ref{thm-odd} (resp. Theorem \ref{thm-even}) by appealing to results on epsilon factors due to G\'erardin-Kutzko \cite{gk80} and Fr\"ohlich-Queyrut \cite{fq73}, and we do this in Section \ref{proof-main}. The proofs of Theorem \ref{thm-odd} and Theorem \ref{thm-even} are rather straightforward in that we write down the explicit invariant vectors in the model for an irreducible representation of $D^\times$ as an induced representation (see Section \ref{prelim}). These invariant vectors are given in Corollary \ref{cor1}, Corollary \ref{cor2}, Corollary \ref{cor3}, Corollary \ref{cor4}, and Corollary \ref{cor5}. Making use of these explicit invariant vectors, the corresponding correlation coefficients are calculated in Section \ref{correlation}. In fact, our methods for writing down the invariant vectors prove much more and we can write down equivariant vectors for a large number of characters of the tori. These results are stated as Proposition \ref{prop1}, Proposition \ref{prop2}, Proposition \ref{prop3}, Proposition \ref{prop4}, and Proposition \ref{prop5}. The number of characters for which we can do so is quantified in Remark \ref{rmk1}, Remark \ref{rmk2}, Remark \ref{rmk3}, Remark \ref{rmk4}, and Remark \ref{rmk5}.    

\section{Representations of $D^\times$}\label{prelim}
Let $D$ be the quaternion division algebra over a $p$-adic field $F$, where we assume that $p$ is odd. Let $\mathcal O_D$ (resp. $\mathcal O_F$) be the ring of integers of $D$ (resp. $F$) and let $\mathcal P_D$ (resp. $\mathcal P_F$) denote the unique maximal ideal of $\mathcal O_D$ (resp. $\mathcal O_F$). Let $\zeta$ be a primitive $(q^2 - 1)$-st root of unity in $D$, where $q$ is the residue cardinality of $F$. We denote by $\mathbb D$ the set of coset representatives of $\mathcal{O}_D/\mathcal{P}_D$ given by $\left\langle \zeta \right\rangle \cup \{ 0\}$. Note that a set of coset representatives of $\mathcal O_F/\mathcal P_F$ is given by 
$\mathbb F = \mathbb D \cap F = \{ \zeta^{i(q+1)}: 0 \leq i \leq q-2 \} \cup \{0\}$. We fix a uniformizer $\varpi$ of $D$ such that $\varpi^2 = \varpi_F$ is a uniformizer of $F$ and $\varpi \zeta \varpi^{-1}= \zeta^q$. Let $\mathrm{Tr}_{\mathbb D/\mathbb F}: \mathbb D \rightarrow \mathbb F$ be the trace map and let $\mathbb T = \{\gamma \in \mathbb D \mid \mathrm{Tr}_{\mathbb D/\mathbb F}(\gamma) = 0\}$. Note that $\mathbb D = \mathbb T \bigoplus \mathbb F$.

The irreducible representations of $D^\times$ are classified in \cite{gg68}. These representations, in fact not just for quaternions but for tame division algebras of any degree, can be constructed in two different ways (see \cite[Chapter 8]{bf83} for more details); via the theory of admissible pairs due to Corwin and Howe \cite{ch77} and via a more elementary and explicit approach due to Koch and Zink \cite{kz80}. For the purposes of this paper, it is the latter approach which is most suitable and we summarize the constructions in \cite{kz80,bf83} and here our exposition closely follows \cite[\S1]{hsy93}. In the quaternionic case, these latter constructions are also there in \cite[\S 3.4, \S 3.5, \S 3.6, \S 4.5]{gk80}.

The construction of an irreducible representation of $D^\times$ involves a character $\chi$ of $E^\times$ where $E$ is a quadratic extension of $F$. Recall that the conductor $f(\chi)$ of $\chi$ is the smallest non-negative integer $\nu$ such that $\rho|_{1 + \mathcal{P}_E^{\nu}} = 1$. Let $\mathrm{Nm}_{E/F}$ and $\mathrm{Tr}_{E/F}$ be the norm and trace maps from $E$ to $F$. A character $\chi$ of $E^\times$ is called minimal if $f(\chi) \leq f(\chi \cdot \phi \circ \mathrm{Nm}_{E/F})$ for all characters $\phi$ of $F^\times$. We also fix an additive character $\psi_0$ of $F$ of conductor $0$ and thus $\psi_0(\mathcal{O}_F) = 1$ but $\psi_0(\mathcal{P}_F^{-1}) \neq 1$. Define $\psi_E = \psi_0 \circ {\rm Tr}_{E/F}$ and $\psi_D = \psi_0 \circ {\rm Tr}_{D/F}$, where $\mathrm{Tr}_{D/F}$ is the reduced trace.

If $\pi$ is an irreducible representation of $D^\times$ then its conductor $f(\pi)$ is the smallest non-negative integer $\nu$ such that $\rho|_{1 + \mathcal{P}_D^{\nu-1}} = 1$. The representation $\pi$ is said to be minimal if $f(\pi) \leq f(\pi \otimes \phi \circ \mathrm{Nm}_{D/F})$, where $\mathrm{Nm}_{D/F}$ is the reduced norm. Suppose $\pi$ arises out of $(E,\chi)$ and is minimal then it turns out that $f(\pi)$ is odd if and only if $E/F$ is ramified. We say that $\pi$ is ramified (resp. unramified) when $E/F$ is ramified (resp. unramified).

\subsection{Ramified representations} For more details of the construction in this case, we refer to \cite[p. 133]{hsy93}. Suppose $K$ is a ramified quadratic extension of $F$ whose ring of integers is $\mathcal O_K$ with its unique maximal ideal $\mathcal P_K$. Let $\psi$ be a minimal character of $K^\times$ with even conductor. We write $f(\psi) = m+1$ where $m$ is an odd integer. In this case, we have $f(\psi|_{F^\times}) \leq f(\psi)/2$. 

There exists $c \in \mathcal P_K^{-m-2}$ such that $\psi(1+x) = \psi_K(cx)$ for $1+x \in 1+\mathcal P_K^{\frac{m+1}{2}}$. Now define the character $\psi_c$ of $1+\mathcal P_D^{\frac{m+1}{2}}$ by $\psi_c(1+x) = \psi_D(cx)$. Note that $\psi_c(1+\mathcal P_D^m) \neq 1$. and $\psi_c(1+\mathcal P_D^{m+1}) = 1$. The inertial subgroup of $\psi_c$ is $H = K^\times(1+\mathcal P_D^{\frac{m+1}{2}})$. Extend $\psi_c$ to $H$ by $\psi^\#(y (1+x)) = \psi(y)\psi_c(1+x)$.

Then the representation \[\pi_\psi = {\rm Ind}^{D^\times}_{H} \psi^\# \]
is irreducible with $f(\pi_\psi) = m+2$ and its central character is $\psi|_{F^\times}$. Note that $\dim \pi_\psi = |D^\times/H| = (q+1)q^{\frac{m-1}{2}}$. Any minimal irreducible representation of $D^\times$ of odd conductor arises this way from one of the two ramified quadratic extensions of $F$.

\subsection{Unramified representations} For more details of the construction in this case, we refer to \cite[p. 126]{hsy93}. Suppose $M$ is the unramified quadratic extension of $F$ whose ring of integers is $\mathcal O_M$ with its unique maximal ideal $\mathcal P_M$. Let $\psi$ be a minimal character of $M^\times$ with conductor $\geq 1$. We write $f(\psi) = m+1$. In this case, we have $f(\psi|_{F^\times}) \leq f(\psi)$. 

There exists $c \in  \mathcal P_M^{-m-1} - \mathcal P_M^{m+1} \cap (F + \mathcal P_M^{-m})$ such that $\psi(1 + x)= \psi_M(cx)$ for each $x \in \mathcal P_M^{[\frac{m+2}{2}]}$. We also have $M=F(c)$. Define the character $\psi_c$ of $1+\mathcal P_D^{m+1}$ by $\psi_c(1+x) = \psi_D(cx)$. Note that $\psi_c(1+\mathcal P_D^{2m}) \neq 1$. and $\psi_c(1+\mathcal P_D^{2m+1}) = 1$.  The inertial subgroup of $\psi_c$ is $H = M^\times(1+\mathcal P_D^m)$.

We define a representation $\theta_\psi^\#$ of $H$ as follows.
\begin{enumerate}
\item
When $m$ is even, note that $M^\times (1 + \mathcal P_D^m) = M^\times (1 + \mathcal P_D^{m+1})$. Now $\psi_c$ is extended by setting  $\theta_\psi^\#(y(1+x)) = \psi(y) \psi_c(1+x)$. 
\item
When $m$ is odd, let $H^{\prime \prime} = F^\times (1 +\mathcal P_M) ( 1 + \mathcal P_D^{m+1})$ and $H^\prime = F^\times (1 +\mathcal P_M) ( 1 + \mathcal P_D^m)$. We first extend $\psi$ to a character $\psi^{\prime \prime}$ of $H^{\prime \prime}$ by setting $\psi^{\prime \prime}(a(1+y)(1 + x)) = \psi (a) \psi(1+y) \psi_c(1+x)$. Now $H^{\prime \prime}$ is a normal subgroup of $H^\prime$ with $H^\prime/H^{\prime \prime}$ being an elementary abelian $p$-group with order $q^2$. Using \cite[Proposition 8.3.3]{bf83}, we have a maximal abelian subgroup $H^\prime_1$ of $H^\prime$ containing $H^{\prime \prime}$ such that extending $\psi^{\prime \prime}$ to a character $\psi^\prime$ of $H^\prime_1$, there exists a uniquely determined irreducible $q$-dimensional representation $\theta_\psi$ of $H^\prime$ defined by $\theta_\psi = \mathrm{Ind}^{H^\prime}_{H_1^\prime} \psi^\prime$, and $\theta_\psi|{H^{\prime \prime}}$ contains $\psi^{\prime \prime}$. Then by \cite[Proposition 8.4.1]{bf83}, this representation $\theta_\psi$ can be extended to an irreducible degree $q$ representation $\theta_\psi^\#$ of $H$ uniquely under the condition  $\mathrm{det}[\theta_\psi^\#|_{\langle \zeta \rangle}]= \psi^q|_{\langle \zeta \rangle}$.
\end{enumerate}

Then the representation \[\pi_\psi = {\rm Ind}^{D^\times}_{H} \theta_\psi^\# \]
is irreducible with $f(\pi_\psi) = 2m+2$ and its central character is $\psi|_{F^\times}$. Note that $\dim \pi_\psi = |D^\times/H| \dim \theta_\psi^\#= 2q^m$. Any minimal irreducible representation of $D^\times$ of even conductor arises this way from the unramified quadratic extension of $F$.

\section{Construction of invariant vectors: Ramified Case}\label{ramified}

Let $K/F$ be a quadratic ramified extension and let $H = K^\times (1+\mathcal P_D^{\frac{m+1}{2}})$, where $m$ is an odd positive integer. Let 
\[\mathcal S = \left\{ (\alpha_1, \alpha_2,\dots, \alpha_{\frac{m-1}{2}}) \mid \alpha_r \in \mathbb F, 1 \leq r \leq \frac{m-1}{2} \right\}.\] 
In what follows, by $S^\prime$ (resp. $S^{\prime \prime}$) we denote the subset of $\mathcal{S}$ consisting of those tuples $\underline{\alpha}=(\alpha_1, \alpha_2,\dots,\alpha_{\frac{m-1}{2}})$ where $\alpha_r =0$ when $r$ is odd (resp. even). 

\subsection{Restriction to the unramified torus}

Let $M$ be the unramified quadratic extension over $F$. Then $M$ is generated by $\zeta$ over $F$. For each integer $i$ with $0 \leq i \leq q$ and $\underline{\alpha}=(\alpha_1, \alpha_2,..., \alpha_{\frac{m-1}{2}}) \in \mathcal S$, let
\[d_{\underline{\alpha}, i}= (1 + \sum\limits_{r=1}^{\frac{m-1}{2}} \alpha_r \zeta^{\frac{q+1}{2}} \varpi^r) \cdot \zeta^i.\]
Then the set $\{d_{\underline{\alpha}, i} \mid \underline{\alpha} \in \mathcal S, 0 \leq i \leq q\}$ forms a set of coset representatives for $D^\times/H$; i.e., 
\[D^\times = \bigsqcup_{\substack{\underline{\alpha} \in \mathcal{S}  \\    0 \leq i \leq q }}    H \cdot d_{\underline{\alpha}, i}.\]
For each pair $(\underline{\alpha},i)$ with $\underline{\alpha} \in \mathcal S$ and $0 \leq i \leq q$, we define a complex valued function on $D^\times$ by
\[f_{\underline{\alpha}, i}(h \cdot d_{\underline{\beta}, j}) = \begin{cases}
\psi^\#(h) &\text{if $\underline{\beta} = \underline{\alpha}$ and $j = i$,} \\
0 &\text{otherwise,}
\end{cases} \]
for $h \in H$.
Note that 
\[\{f_{\underline{\alpha}, i} \mid \underline{\alpha} \in \mathcal S, 0 \leq i \leq q \}\] forms a basis for the irreducible representation $\pi_\psi=\mathrm{Ind}^G_{H} \psi^\#$.

\begin{proposition}\label{prop1}
Let $\mu: M^\times \rightarrow \mathbb C^\times$ be a smooth character of $M^\times$ with $\mu|_{H \cap M^\times}=\psi^\#|_{H \cap M^\times}$. Then $\mu$ occurs with multiplicity one in the restriction of $\pi_\psi$ to $M^\times$ and the vector 	
\[v_{\mu}= \sum\limits_{\substack{\underline{\alpha} \in S^\prime  \\    0 \leq i \leq q }} \mu(d_{\underline{\alpha}, i}) f_{\underline{\alpha}, i} \]
lies in the $\mu$-isotypic component of the representation $\pi_\psi$.
\end{proposition}

\begin{proof}
The first part of the assertion is just elementary Mackey theory. We show that $v_\mu$ is a $\mu$-equivariant vector. To this end, fix $x=h \cdot d_{\underline{\delta}, i_0} \in D^\times$, where $h \in H$, $\underline{\delta}=(\delta_1,\dots,\delta_{\frac{m-1}{2}}) \in \mathcal{S}$, and $i_0 \in \{0,1,\dots,q\}$. Clearly,
\[v_{\mu}(x) = \sum\limits_{\substack{\underline{\alpha} \in S^\prime  \\    0 \leq i \leq q }}\mu(d_{\underline{\alpha}, i}) f_{\underline{\alpha}, i}(h \cdot d_{\underline{\delta}, i_0}) = 
\begin{cases}
\psi^\#(h)\mu(d_{\underline{\delta}, i_0})  &\text{if $\underline{\delta} \in S^\prime$,}  \\
0 &\text{otherwise.}\end{cases}\]
If $z \in M^\times$ then we write 
\[z= a \cdot \zeta^{j_0} \cdot (1 + \sum\limits_{\substack{r \geq 2  \\    r \text{ is even} }}\beta_r \cdot  \zeta^{\frac{q+1}{2}} \cdot \varpi^r ) \]
for some $a \in F^\times$, $j_0 \in \{0,1,\dots,q\}$, and $\beta_r \in \mathbb F$ for all $r$. Then,
\begin{align*}
\pi_\psi(z) v_{\mu}(x) &= v_{\mu}(xz) \\
&= v_{\mu}\left(ah\left(1 + \sum\limits_{r=1}^{\frac{m-1}{2}} \delta_r \zeta^{\frac{q+1}{2}} \varpi^r \right) \cdot \left(1 + \sum\limits_{\substack{r \geq 2  \\    r \text{ is even} }}\beta_r \cdot  \zeta^{\frac{q+1}{2}} \cdot \varpi^r \right) \zeta^{i_0 + j_0} \right).
\end{align*}
Observe that the element 
\[\textbf{y} = \left(1 + \sum\limits_{r=1}^{\frac{m-1}{2}} \delta_r \zeta^{\frac{q+1}{2}} \varpi^r \right) \cdot \left(1 + \sum\limits_{\substack{r \geq 2  \\    r \text{ is even} }}\beta_r \cdot  \zeta^{\frac{q+1}{2}} \cdot \varpi^r \right) \zeta^{i_0 + j_0}\]
lies in the support of $v_\mu$ if and only if $\delta_r = 0$ for all odd integers $r \in \{1,\dots,\frac{m-1}{2} \}$. Thus, $v_\mu(\textbf{y})=0$ if $\underline{\delta} \notin S^\prime$. 

Now assume that $\underline{\delta} \in S^\prime$. Let $m^\prime$ be the largest even integer in $\{1,2,\dots,\frac{m-1}{2} \}$. Then, the element 
\[\textbf{y} = \left(1 + \sum\limits_{\substack{r = 2  \\    r \text{ is even} }}^{m^\prime} \delta_r \zeta^{\frac{q+1}{2}} \varpi^r \right) \cdot \left(1 + \sum\limits_{\substack{r \geq 2  \\    r \text{ is even} }}\beta_r \cdot  \zeta^{\frac{q+1}{2}} \cdot \varpi^r \right) \zeta^{i_0 + j_0}\]
can be written as 
\[\textbf{y}= b_0 \cdot h_0 \cdot d_{\underline{\delta}^\prime, i_0 + j_0}\] 
where
\[b_0 =  1 + \sum\limits_{\substack{r = 4  \\    r \text{ is even} }}^{m^\prime} \xi_r  \cdot \varpi^r \] 
with
\[\xi_r = \left(\sum\limits_{\substack{ t = 2  \\    t \text{ is even} }}^{r-2} \delta_t \beta_{r-t} \right) \zeta^{q+1} \]
for $4 \leq r \leq m^\prime$, and
\[\underline{\delta}^\prime = (\delta_1^\prime,...,\delta_{\frac{m-1}{2}}^\prime) \in S^\prime\]
with
\[
\delta_r^\prime = \delta_r + \beta_r - \sum\limits_{\substack{t = 2  \\    t \text{ is even} }}^{r-2} \delta_t^\prime \xi_{r-t} 
\]
for even $r \leq m^\prime$, and
\[h_0 = 1 + \sum\limits_{\substack{r = m^\prime + 2  \\    r \text{ is even}}}^{m-1} (\gamma_r \zeta^{\frac{q+1}{2}} + \xi_r)\varpi^r +\cdots \]
with
\[\gamma_r = \beta_r - \sum\limits_{\substack{ t = 2  \\    t \text{ is even} }}^{m^\prime} \delta_t^\prime \xi_{r-t} - \sum\limits_{\substack{ t =m^\prime +2  \\    t \text{ is even} }}^{r-2} \gamma_t \xi_{r-t}  \in \mathbb F,\]
for $r \geq m^\prime+2$, and
\[\xi_r = \left(\sum\limits_{\substack{ t = 2  \\    t \text{ is even} }}^{m^\prime} \delta_t \beta_{r-t} - \sum\limits_{\substack{ t = m^\prime+2  \\    t \text{ is even} }}^{r-2} \delta_{r-t}^\prime \gamma_t \right)\zeta^{q+1} 
- \sum\limits_{\substack{ t = m^\prime+2  \\    t \text{ is even} }}^{r-2} \xi_t \xi_{r-t} \]
for $m^\prime+2 \leq r \leq m-1$. 

Observe that $b_0 \in F^\times$ and $h_0 \in H \cap M^\times$. By our assumption on $\mu$, we have $\psi^\#(h_0) = \mu(h_0)$. Recall our assumption that $\underline{\delta} \in \mathcal S^\prime$. Now, 
\begin{align*}
 \pi_\psi(z) v_{\mu}(x) &= v_\mu \left(ah \cdot b_0 \cdot h_0 \cdot d_{\underline{\delta^\prime}, i_0 + j_0} \right) \\
&= \psi^\#(h) \psi (a \cdot b_0) \psi^\#(h_0) \mu(d_{\underline{\delta^\prime}, i_0 + j_0})  \\
&= \psi^\#(h) \mu(a \cdot b_0 \cdot h_0) \mu(d_{\underline{\delta^\prime}, i_0 + j_0})  \\
&= \psi^\#(h) \mu(z \cdot d_{\underline{\delta}, i_0}) \\
&= \mu(z) v_{\mu}(x).
\end{align*}
Note that we have used the identity $a \cdot b_0 \cdot h_0 \cdot d_{\underline{\delta^\prime}, i_0 + j_0} =  d_{\underline{\delta}, i_0} \cdot z$ in the last but one step. Thus,
\[\pi_\psi(z) v_{\mu}= \mu(z) v_{\mu}\]
for $z \in M^\times$.
\end{proof}

\begin{remark}\label{rmk1}
Observe that \[H \cap M^\times = F^\times (1 + \mathcal P_M^{[\frac{m-1}{4}]+1}).\] 
and thus we are able to exhibit $\mu$-equivariant vectors in $\pi_\psi$ for $(q+1)q^{[\frac{m-1}{4}]}$ many characters $\mu$.
\end{remark}

If $\pi_\psi$ has trivial central character then the trivial character of $M^\times$ always occurs in the restriction $\pi_\psi$ to $M^\times$. Putting $\mu=1$ in the above proposition, we get 

\begin{corollary}\label{cor1}
If $\pi_\psi$ has trivial central character, an $M^\times$-fixed vector in $\pi_\psi$ is given by
\[v_M = \sum\limits_{\substack{\underline{\alpha} \in S^\prime  \\    0 \leq i \leq q }} f_{\underline{\alpha}, i}. \]
\end{corollary}

\subsection{Restricted to ramified tori}

Let $K$ and $L$ be the two ramified extensions of $F$. The representation $\pi_\psi$ is associated to the extension $K$. We need to consider its restriction to $K$ and to $L$. We take $K = F(\varpi)$ and $L=F(\zeta\varpi)$.

\subsubsection{Restriction to $K^\times$}

It is easy to observe that the character $\psi$ occurs with multiplicity one in the restriction $\pi_\psi$ to $K^\times$ and a $\psi$-fixed vector is given by $f_{\underline{0}, 0}$. Similarly, when $g=\zeta^\frac{q+1}{2} \in D^\times$, the character $\psi^g$ occurs with multiplicity one in the restriction of $\pi_\psi$ and a $\psi^g$-fixed vector is given by $f_{\underline{0}, \frac{q+1}{2}}$. In the proposition below, we exhibit many more equivariant vectors.

For each integer $i$ with $0 \leq i \leq q$ and $\underline{\alpha}=(\alpha_1, \alpha_2, \dots,\alpha_{\frac{m-1}{2}}) \in \mathcal{S}$, we define
\[d_{\underline{\alpha}, i}^\prime= (1 + \sum\limits_{r=1}^{\frac{m-1}{2}} \alpha_r \zeta^{i(1-q)} \varpi^r) \cdot \zeta^i.\]
For each pair $(\underline{\alpha},i)$ with $\underline{\alpha} \in \mathcal{S}$ and $0 \leq i \leq q$, we define the corresponding functions on $D^\times$ by 
\[f^\prime_{\underline{\alpha}, i}(h \cdot d^\prime_{\underline{\beta}, j}) = \begin{cases}
\psi^\#(h) &\text{if $\underline{\beta} = \underline{\alpha}$ and $j = i$,} \\
0 &\text{otherwise,}
\end{cases} \]
for $h \in H$.

\begin{proposition}\label{prop2}
For $1 \leq i \leq \frac{q-1}{2}$, pick the element $g=\zeta^i \in D^\times$ and consider a smooth character $\kappa_g$ of $K^\times$ such that on $\kappa_g|_{g^{-1}Hg \cap K^\times} = \psi^{{\#}^g}$. Then the character $\kappa_g$ occurs with multiplicity one in the restriction of $\pi_\psi$ to $K^\times$ and $\kappa_g|_{F^\times} = \psi|_{F^\times}$. Let $\ell \in \{0,1\}$ be such that $\kappa_g(\varpi) = (-1)^\ell \psi(\varpi)$. The vector
\[v_{\kappa_g} = \sum\limits_{\underline{\alpha} \in S^{\prime \prime} } \kappa_g(1 + \sum\limits_{r} \alpha_r \varpi^r) \left[ f_{\underline{\alpha}, i}^\prime + (-1)^\ell \psi(\zeta^{(1-i)(q+1)})  f_{\underline{\alpha}, q+1-i}^\prime \right]\] 
lies in the $\kappa_g$-isotypic component of the representation $\pi_\psi$.
\end{proposition}

\begin{proof}
The first assertion is straightforward Mackey theory. We now show that $v_{\kappa_g}$ is a $\kappa_g$-invariant vector. To this end, observe that $v_{\kappa_g}(x) = 0$ unless fix $x = h \cdot d_{\underline{\delta}, i}^\prime$ or 
$x = h \cdot d_{\underline{\delta}, {q+1-i}}^\prime$ for $h \in H$, $\underline{\delta}=(\delta_1,\dots,\delta_{\frac{m-1}{2}}) \in \mathcal S^{\prime\prime}$. Clearly,
\[
v_{\kappa_g}(x) = 
\begin{cases}
\psi^\#(h) \cdot \kappa_g(1 + \sum\limits_{r} \delta_r \varpi^r)  &\text{if $x = h \cdot d_{\underline{\delta}, i}^\prime $,} \\
(-1)^\ell \psi(\zeta^{(1-i)(q+1)}) \psi^\#(h) \cdot \kappa_g(1 + \sum\limits_{r} \delta_r \varpi^r)  &\text{if $x = h \cdot d_{\underline{\delta}, q+1-i}^\prime $,} \\
0 &\text{otherwise.}
\end{cases}
\]
We write $z \in K^\times$ as 
\[z= a \cdot \varpi^{j_0} \cdot (1 + \sum\limits_{\substack{j \geq 1  \\    j \text{ is odd} }}\beta_j \cdot  \varpi^j )\]
for some $a \in F^\times$, $j_0 \in \{0,1\}$, and $\beta_j \in \mathbb F$ for all $r$. Clearly, when $a \in F^\times \subset K^\times $, we have 
\[\pi_\psi(a) v_{\kappa_g}(x)=v_{\kappa_g}(xa)=\psi(a) \cdot v_{\kappa_g}(x) = \kappa_g(a) v_{\kappa_g}(x).\]
Now we check the action of the uniformizer $\varpi \in K^\times$ and elements of the form 
\[z^\prime = 1 + \sum\limits_{\substack{j \geq 1  \\    j \text{ is odd} }}\beta_j \cdot  \varpi^j \]
with $\beta_r \in \mathbb F$ for all $r$. Without loss of generality we may take $x =  h \cdot d_{\underline{\delta}, i}^\prime$, as the proof in the case $x =  h \cdot d_{\underline{\delta}, q+1-i}^\prime$ being similar. Then, 
\begin{align*}
\pi_\psi(\varpi) v_{\kappa_g}(x) &= v_{\kappa_g}(h \cdot d_{\underline{\delta}, i}^\prime  \cdot \varpi) \\
&= v_{\kappa_g}(h \cdot \varpi \cdot \zeta^{(i-1)(q+1)} \cdot d_{\underline{\delta}, q+1-i}^\prime) \\
&= \psi^\#(h) \cdot \psi(\varpi) \cdot \psi(\zeta^{(i-1)(q+1)}) \cdot \kappa_g(1 + \sum\limits_{r} \delta_r \varpi^r)  (-1)^\ell \psi(\zeta^{(1-i)(q+1)}) \\
&= (-1)^\ell\psi(\varpi)  \cdot  \psi^\#(h) \kappa_g(1 + \sum\limits_{r} \delta_r \varpi^r)  \\
&= \kappa_g(\varpi) \cdot v_{\kappa_g}(x).
\end{align*}
Now we check the action of $z^\prime$. We have 
\begin{align*}
\pi_\psi(z^\prime) v_{\kappa_g}(x) &= v_{\kappa_g}(h \cdot d_{\underline{\delta}, i}^\prime \cdot (1 + \sum\limits_{\substack{j \geq 1  \\    j \text{ is odd} }}\beta_j \cdot \varpi^j )) \\
&= v_{\kappa_g}\left( h (1 + \sum\limits_{\substack{ r = 1  \\  r \text{ is odd} }}^{\frac{m-1}{2}} \delta_r \zeta^{i(1-q)} \varpi^r)  \cdot (1 + \sum\limits_{\substack{j \geq 1  \\    j \text{ is odd} }}\beta_j  \zeta^{i(1-q)} \varpi^j ) \zeta^i \right).
\end{align*}
Let $m^*$ be the largest odd integer in $\{1,2,...,\frac{m-1}{2}\}$. Then the element 
\[\textbf{y} = (1 + \sum\limits_{\substack{r = 1  \\    r \text{ is odd} }}^{m^*} \delta_r \zeta^{i(1-q)} \varpi^r) (1 + \sum\limits_{\substack{j \geq 1  \\    j \text{ is odd} }}\beta_j  \zeta^{i(1-q)} \varpi^j ) \zeta^i\]
can be written as 
\[\textbf{y } = b_0 \cdot h_0 \cdot d^\prime_{\underline{\delta}^{\prime},i}\] 
where,
\[b_0 = 1 + \sum\limits_{\substack{r = 2 \\    r \text{ is even} }}^{m^*+1} \xi_r  \varpi^r \]
with 
\[\xi_r = \sum\limits_{\substack{t=1  \\    t \text{ is odd} }}^{r-1} \delta_{r-t}\beta_t \]
for each even $r \leq m^*+1$, 
and
\[\underline{\delta}^{\prime} = (\delta_1^{\prime},...,\delta_{\frac{m-1}{2}}^{\prime}) \in S^{\prime\prime}\]
with 
\[\delta_r^{\prime} = \delta_r + \beta_r  - \sum\limits_{\substack{t =1 \\    t \text{ is odd} }}^{r-2} \delta_t^{\prime} \xi_{r-t},\]
for each odd $r \leq m^*$, 
and
\[h_0 = 1 + \sum\limits_{\substack{j =m^*+2  \\    j \text{ is odd} }}^{m} \epsilon_j \zeta^{i(1-q)}\varpi^j + \sum\limits_{\substack{k =m^* +3 \\    k \text{ is even} }}^{m-1} \xi_k \varpi^k +  \cdots \]
such that
\[\epsilon_j = \beta_j -\sum\limits_{\substack{t=1  \\    t \text{ is odd} }}^{m^*} \delta_t^\prime \xi_{j-t}- \sum\limits_{\substack{t=m^* +2  \\    t \text{ is odd} }}^{j-2} \epsilon_t \xi_{j-t}, \]
for odd $j \geq \frac{m+1}{2}$,
and
\[\xi_k = \sum\limits_{\substack{t=1  \\    t \text{ is odd } }}^{m^*}\delta_t \beta_{k-t} -\sum\limits_{\substack{t=m^*+3  \\    t \text{ is even} }}^{k-2} \xi_t \xi_{k-t} - \sum\limits_{\substack{t=m^*+2  \\    t \text{ is odd} }}^{k-1} \delta_{k-t}^\prime \epsilon_{t},\]
for even $k \geq \frac{m+1}{2}$.
Denote 
\[\tilde{h} = g^{-1}h_0 g = 1 + \sum\limits_{\substack{j =m^*+2  \\    j \text{ is odd} }}^{m} \epsilon_j \varpi^j + \sum\limits_{\substack{k =m^* +3 \\    k \text{ is even} }}^{m-1} \xi_k \varpi^k +  \cdots.\]
Note that $\tilde{h} \in g^{-1}Hg \cap K^\times$. We have
\begin{align*}
 \pi_\psi(z^\prime) v_{\kappa_g}(x) &= v_{\kappa_g}(h \cdot b_0 \cdot h_0 \cdot d_{\underline{\delta}^\prime,i}^\prime) \\
&= \psi^\#(h) \cdot \psi(b_0) \cdot \psi^{\#^g}(\tilde{h}) \cdot \kappa_g(1 + \sum\limits_{\substack{r = 1  \\    r \text{ is odd} }}^{m^*} \delta_r^{\prime} \varpi^r) ,\\
&= \psi^\#(h) \cdot \kappa_g \left(b_0 \cdot \tilde{h} \cdot (1 + \sum\limits_{\substack{r = 1  \\    r \text{ is odd} }}^{m^*} \delta_r^{\prime} \varpi^r)\right).\\
\end{align*}
Observe that 
\begin{align*}
z^\prime \cdot (1 + \sum\limits_{\substack{r = 1  \\    r \text{ is odd} }}^{m^*} \delta_r \varpi^r) 
&= (1 + \sum\limits_{\substack{j \geq 1  \\    j \text{ is odd} }} \beta_j  \varpi^j ) \cdot (1 + \sum\limits_{\substack{r = 1  \\    r \text{ is odd} }}^{m^*} \delta_r \varpi^r) \\
&= b_0 \cdot (1 + \sum\limits_{\substack{j =m^*+2  \\    j \text{ is odd} }}^{m} \epsilon_j \varpi^j + \sum\limits_{\substack{k =m^* +3 \\    k \text{ is even} }}^{m-1} \xi_k \varpi^k +  \sum\limits_{t \geq m+1} \gamma_t \varpi^t) \cdot (1 + \sum\limits_{\substack{r = 1  \\    r \text{ is odd} }}^{m^*} \delta_r^{\prime} \varpi^r) 
\end{align*}
for some $\gamma_t \in \mathbb F$ for $t \geq m+1$. Since $ \kappa_g|_{(1 + \mathcal{P}_K^{m+1})}=1$, we have 
\begin{align*}
\kappa_g \left(b_0 \cdot \tilde{h} \cdot (1 + \sum\limits_{\substack{r = 1  \\    r \text{ is odd} }}^{m^*} \delta_r^{\prime} \varpi^r)\right) 
&=  \kappa_g \left( b_0 \cdot (1 + \sum\limits_{\substack{j =m^*+2  \\    j \text{ is odd} }}^{m} \epsilon_j \varpi^j + \sum\limits_{\substack{k =m^* +3 \\    k \text{ is even} }}^{m-1} \xi_k \varpi^k)\cdot (1 + \sum\limits_{\substack{r = 1  \\    r \text{ is odd} }}^{m^*} \delta_r^{\prime} \varpi^r)\right) \\
&= \kappa_g \left (z^\prime \cdot (1 + \sum\limits_{\substack{r = 1  \\    r \text{ is odd} }}^{m^*} \delta_r \varpi^r) \right) = \kappa_g(z^\prime)  \kappa_g(1 + \sum\limits_{r} \delta_r \varpi^r). 
\end{align*}
Therefore,
\begin{align*}
 \pi_\psi(z^\prime) v_{\kappa_g}(x) &= \kappa_g(z^\prime) \cdot \psi^\#(h) \kappa_g(1 + \sum\limits_{r} \delta_r \varpi^r) \\
 &= \kappa_g(z^\prime) \cdot v_{\kappa_g}(x).
\end{align*}
Hence, $\pi_\psi(z) \cdot v_{\kappa_g}= \kappa_g(z) \cdot v_{\kappa_g}$ for all $z \in K^\times$.
\end{proof}

\begin{remark}\label{rmk2}
For each $g=\zeta^i$, where $1 \leq i \leq \frac{q-1}{2}$, the number of characters $\kappa_g: K^\times \rightarrow \mathbb C$ satisfying $\kappa_g |_{g^{-1}Hg \cap K^\times} = \psi^{{\#}^g}$ is 
\[ \left| \frac{K^\times}{F^\times (1 + \mathcal P_K^{\frac{m+1}{2}})} \right| = 2q^{\frac{m-1}{2} [\frac{m-1}{4}]}.\] 
Thus, in Proposition \ref{prop2}, we have exhibited equivariant vectors in $\pi_\psi$ for 
\[\frac{q-1}{2} \times 2q^{\frac{m-1}{2} - \left[\frac{m-1}{4} \right]} = (q-1)q^{\frac{m-1}{2} - \left[\frac{m-1}{4} \right]}\]
many characters $\kappa_g$.
\end{remark}

If $q \equiv 3 \mod 4$ and $\psi|_{F^\times} = 1$ then the trivial character $\textbf{1}_{K^\times}$ of $K^\times$ occurs in the restriction $\pi_\psi$ to $K^\times$. We have the following result as a corollary to the above proposition. 

\begin{corollary}\label{cor2}
Suppose $q \equiv 3 \mod 4$ and write $q+1 = 4t_0$. Assume that $\pi_\psi$ has trivial central character. Then
\[v_K = \sum\limits_{\underline{\alpha} \in S^{\prime \prime} } \left[ f_{\underline{\alpha}, t_0} + \psi(\varpi) f_{\underline{\alpha}, 3t_0} \right]\]
is a $K^\times$-fixed vector in $\pi_\psi$.
\end{corollary}

\begin{proof}
If $\pi_\psi$ has trivial central character then $\psi(\varpi) = \pm 1$. Since $q+1=4t_0$, we have 
\[\zeta^{t_0 (1-q)} = \eta \cdot \zeta^{\frac{q+1}{2}},\]
where $\eta=-\zeta^{(t_0-1)(q+1)} \in \mathbb F$. Note that $\zeta^{-t_0 (1-q)} = -\eta \cdot  \zeta^{\frac{q+1}{2}}$. Therefore, for each $\underline{\alpha}=(\alpha_1, \alpha_2, \dots, \alpha_{\frac{m-1}{2}}) \in \mathcal{S}^{\prime \prime}$, we have
\[f_{\underline{\alpha}, t_0}^\prime = f_{\underline{\alpha \eta}, t_0},\]
where $\underline{\alpha \eta }=(\alpha_1 \eta, \alpha_2 \eta, \dots, \alpha_{\frac{m-1}{2}}\eta) \in \mathcal{S}^{\prime \prime}$, and
\[f_{\underline{\alpha}, q+1-t_0}^\prime = f_{-\underline{\alpha \eta}, 3t_0},\]
where $-\underline{\alpha \eta} = (-\alpha_1 \eta, -\alpha_2 \eta,\dots, -\alpha_{\frac{m-1}{2}}\eta) \in \mathcal{S}^{\prime \prime}$.
When $g = \zeta^{t_0}$ and $\psi|_{F^\times} = 1$, we have $\kappa_g|_{g^{-1}Hg \cap K^\times}= \psi^{\#^g}=1$ because $H \cap g K^\times g^{-1} \subset \mathrm{Ker}(\psi^\#)$. Now with $g=\zeta^{t_0}$ and $\kappa_g=1$ in Proposition \ref{prop2},  we get 
\begin{align*}
v_K &= \sum\limits_{\underline{\alpha} \in S^{\prime \prime} } \kappa_{g}(1 + \sum\limits_{r} \alpha_r \varpi^r) \left[ f_{\underline{\alpha}, t_0}^\prime + (-1)^\ell \psi(\zeta^{(1-t_0)(q+1)})  f_{\underline{\alpha}, q+1-t_0}^\prime \right] \\
&= \sum\limits_{\underline{\alpha} \in S^{\prime \prime} } \left[ f_{\underline{\alpha}, t_0}^\prime + \psi(\varpi)   f_{\underline{\alpha}, q+1-t_0}^\prime \right] \\
&= \sum\limits_{\underline{\alpha} \in S^{\prime \prime} } \left[ f_{\underline{\alpha \eta}, t_0} + \psi(\varpi)  f_{-\underline{\alpha \eta}, 3t_0} \right] \\
& =\sum\limits_{\underline{\alpha} \in S^{\prime \prime} }  \left[f_{\underline{\alpha}, t_0} + \psi(\varpi)   f_{\underline{\alpha}, 3t_0} \right].
\end{align*}
\end{proof}

\begin{remark}\label{extra-k}
We can state Corollary \ref{cor2} slightly more generally as follows. Let $s$ be an integer such that $-\frac{q-3}{4} \leq s \leq \frac{q+1}{4}$ and let $K_s=F(\zeta^{2s} \varpi)$. Then, the vector 
\[v_{K_s}= \sum\limits_{\underline{\alpha} \in S^{\prime \prime}} [f_{\underline{\alpha}, t_0-s} + \psi(\varpi) f_{\underline{\alpha}, 3t_0-s}]\]
is a $K_s^\times$-fixed vector in $\pi_\psi$.
\end{remark}

\subsubsection{Restriction to $L^\times$}

Let $L$ be the ramified quadratic field extension over $F$ which is not conjugate to $K = F(\varpi)$. We take $L = F(\zeta \varpi)$. 

For each integer $i$ with $0 \leq i \leq q$ and $\underline{\alpha}=(\alpha_1, \alpha_2, \dots \alpha_{\frac{m-1}{2}}) \in \mathcal{S}$, we define
\[d_{\underline{\alpha}, i}^{\prime\prime} = (1 + \sum\limits_{r=1}^{\frac{m-1}{2}} \alpha_r \zeta^{1+i(1-q)} \varpi^r) \cdot \zeta^i.\] 
For each pair $(\underline{\alpha},i)$ with $\underline{\alpha} \in \mathcal{S}$ and $0 \leq i \leq q$, we define the corresponding functions on $D^\times$ by 
\[f^{\prime\prime}_{\underline{\alpha}, i}(h \cdot d^{\prime\prime}_{\underline{\beta}, j}) = \begin{cases}
\psi^\#(h) &\text{if $\underline{\beta} = \underline{\alpha}$ and $j = i$,} \\
0 &\text{otherwise.}
\end{cases} \]
Note that these functions lie in the induced space of $\pi_\psi = \mathrm{Ind}^G_{H} \psi^\#$.

\begin{proposition}\label{prop3}
For $1 \leq i \leq \frac{q+1}{2}$, pick the element $g = \zeta^i \in D^\times$ and consider a smooth character $\lambda_g$ of $L^\times$ such that on $\lambda_g|_{g^{-1}Hg \cap L^\times} = \psi^{{\#}^g}$. Then the character $\lambda_g$ occurs with multiplicity one in the restriction of $\pi_\psi$ to $L^\times$ and $\lambda_g|_{F^\times} = \psi|_{F^\times}$. Let $\ell \in \{0,1\}$ be such that  and $\lambda_g(\varpi_L) = \lambda_g(\zeta \varpi) = (-1)^\ell \sqrt{\psi(\zeta^{q+1})} \psi(\varpi)$. The vector 
\[v_{\lambda_g} = \sum\limits_{\underline{\alpha} \in S^{\prime \prime} } \lambda_g(1 + \sum\limits_{r} \alpha_r \zeta \varpi^r) \left[ f_{\underline{\alpha}, i}^{\prime\prime} + (-1)^\ell  
\psi(\zeta^{-i(q+1)})\sqrt{\psi(\zeta^{q+1})}
f_{\underline{\alpha}, q-i}^{\prime \prime} \right]\]
lies in the $\lambda_g$-isotypic component of the representation $\pi_\psi$.
\end{proposition}

\begin{proof}
The first assertion is straightforward Mackey theory. To show that $v_{\lambda_g}$ is a $\lambda_g$-invariant vector, we write each element of $L^\times$ in the form 
\[z= a \cdot (\zeta \varpi)^{j_0} \cdot (1 + \sum\limits_{\substack{j \geq 1  \\    j \text{ is odd} }}\beta_j \cdot  \zeta \varpi^j )\]
for some $a \in F^\times$, $j_0 \in \{0,1\}$, and $\beta_j \in \mathbb F$ for all $j$. Then the proof proceeds exactly along the lines of the proof of Proposition \ref{prop2}. 	
\end{proof}	

\begin{remark}\label{rmk3}
For each $g=\zeta^i$, where $1 \leq i \leq \frac{q+1}{2}$, the number of characters $\lambda_g: L^\times \rightarrow \mathbb C$ satisfying $\lambda_g|_{g^{-1}Hg \cap L^\times} = \psi^{{\#}^g}$ is 
\[\left| \frac{L^\times}{F^\times (1 + \mathcal P_L^{\frac{m+1}{2}})} \right| = 2q^{\frac{m-1}{2} - \left[\frac{m-1}{4} \right]}.\]
Thus, in Proposition \ref{prop3}, we have exhibited equivariant vectors for 
\[\frac{q+1}{2} \times 2q^{\frac{m-1}{2} - \left[\frac{m-1}{4} \right]} = (q+1)q^{\frac{m-1}{2} - \left[\frac{m-1}{4} \right]}\] many characters $\lambda_g$. 
\end{remark}

If $q \equiv 1 \mod 4$ and $\psi|_{F^\times} = 1$ then the trivial character $\textbf{1}_{L^\times}$ of $L^\times$ occurs in the restriction $\pi_\psi$ to $L^\times$.  We have the following result as a corollary to the above proposition.

\begin{corollary}\label{cor3}
Suppose $q \equiv 1 \mod 4$ and write $q-1 = 4r_0$. Assume that $\pi_\psi$ has trivial central character. Then
\[v_L = \sum\limits_{\underline{\alpha} \in S^{\prime \prime} } \left[ f_{\underline{\alpha}, r_0} + \psi(\varpi) f_{\underline{\alpha}, 3r_0+1} \right]\]
is an $L^\times$-fixed vector in $\pi_\psi$.
\end{corollary}

\begin{proof}
If $q-1 = 4r_0$ then 
\[\zeta^{1+ r_0 (1-q)} = \eta^\prime \cdot \zeta^{\frac{q+1}{2}},\]
where $\eta^\prime = \zeta^{-r_0(q+1)}$. Note that $\zeta^{q+ r_0 (q-1)} = -\eta^\prime \cdot  \zeta^{\frac{q+1}{2}}$. Therefore, for each $\underline{\alpha}=(\alpha_1, \alpha_2, \dots, \alpha_{\frac{m-1}{2}}) \in \mathcal{S}^{\prime \prime}$, we have
\[f_{\underline{\alpha}, r_0}^{\prime\prime} = f_{\underline{\alpha \eta}^\prime, r_0},\]
where $\underline{\alpha \eta}^\prime = (\alpha_1 \eta^\prime, \alpha_2 \eta^\prime, \dots, \alpha_{\frac{m-1}{2}}\eta^\prime) \in \mathcal{S}^{\prime \prime}$, and 
\[f_{\underline{\alpha}, q-r_0}^{\prime \prime} = f_{-\underline{\alpha \eta}^\prime, 3r_0+1},\]
where $-\underline{\alpha \eta}^\prime = (-\alpha_1 \eta^\prime, -\alpha_2 \eta^\prime, \dots, -\alpha_{\frac{m-1}{2}}\eta^\prime) \in \mathcal{S}^{\prime \prime}$. The rest of the proof follows from Proposition \ref{prop3}
\end{proof}

\begin{remark}\label{extra-l}
We can state Corollary \ref{cor3} slightly more generally as follows. Let $s$ be an such that $-\frac{q-1}{4} \leq s \leq \frac{q+1}{4}$ and let $L_s=F(\zeta^{2s+1} \varpi)$. Then, the vector 
\[v_{L_s}= \sum\limits_{\underline{\alpha} \in S^{\prime \prime}} [f_{\underline{\alpha}, r_0-s} + \psi(\varpi) f_{\underline{\alpha}, 3r_0-s+1}]\]
is a $L_s^\times$-fixed vector $\pi_\psi$.
\end{remark}

\section{Construction of invariant vectors: Unramified Case}\label{unramified}

Recall that we are working with $M = F(\zeta), K = F(\varpi)$, and $L = F(\zeta\varpi)$.

Let
\[\mathcal{R} =\{ (\alpha_1, 0, \alpha_3, 0, \dots, \alpha_{m-1}) \mid \alpha_r \in \mathbb D, 1 \leq r \leq m-1\}.\]
Let $\mathcal{R}^\prime$ denote the subset of $\mathcal{R}$ consisting of those tuples $\underline{\alpha}=(\alpha_1, 0, \alpha_3, 0, \dots, \alpha_{m-1})$ such that $\alpha_r \in \mathbb F$ for all $r$. 

For each integer $i$ with $i \in \{0,1\}$ and $\underline{\alpha}=(\alpha_1, 0, \alpha_3, 0, \dots, \alpha_{m-1}) \in \mathcal{R}$, we define
\[\Delta_{i,\underline{\alpha}} = \varpi^i \cdot (1 + \sum\limits_{r=1}^{m-1} \alpha_r  \varpi^r).\]
Then
\[\{\Delta_{i,\underline{\alpha}} \mid \underline{\alpha} \in \mathcal{R}, i= 0, 1\}\]
forms a set of right coset representatives for $D^\times/H$, where $H = M^\times(1+\mathcal P_D^m)$. 

We have
\[\pi_\psi = {\rm Ind}_H^{D^\times} \theta_\psi^\#,\]
where $\theta_\psi^\#$ is a one dimensional character when $m$ is even and is a $q$-dimensional representation when $m$ is odd. For convenience, let $W$ denote the space of the representation $\theta_\psi^\#$ (in either case). 

When $m$ is even, to each pair $(i,\underline{\alpha})$ with $\underline{\alpha} \in \mathcal{R}$ and $i \in \{0,1\}$, we attach a function 
$\phi_{i,\underline{\alpha}} : D^\times \rightarrow \mathbb C \text{ supported on } H \cdot \Delta_{i,\underline{\alpha}}$ 
such that $\phi_{i,\underline{\alpha}}$ takes value $\psi^\#(h)$ on the element $h \cdot \Delta_{i,\underline{\alpha}}$ for each $h \in H$.
 
We also define, for $\underline{\alpha}=(\alpha_1, 0, \alpha_3, 0, \dots, \alpha_{m-1}) \in \mathcal{R}^\prime$,
\[\Delta^\prime_{i,\underline{\alpha}}=(\zeta \varpi)^i \cdot (1 + \sum\limits_{r=1}^{m-1} \alpha_r \zeta \varpi^r),\]
and corresponding functions
$\phi^\prime_{i,\underline{\alpha}} : D^\times \rightarrow \mathbb C \text{ supported on } H \cdot \Delta^\prime_{i,\underline{\alpha}}$ 
such that $\phi^\prime_{i,\underline{\alpha}}$ takes value $\psi^\#(h)$ on the element $h \cdot \Delta^\prime_{i,\underline{\alpha}}$ for each $h \in H$. 

When $m$ is odd, recall that
$H^{\prime \prime} = F^\times (1 +\mathcal P_M) ( 1 + \mathcal P_D^{m+1})$ and $H^\prime = F^\times (1 +\mathcal P_M) ( 1 + \mathcal P_D^m)$. 
Let
\[H_1^\prime = F^\times (1 +\mathcal{P}_M) (1+\mathcal P_L^m)(1 +\mathcal{P}_D^{m+1}).\] 
Note that the set $\mathcal{D} = \{1+ \alpha \varpi^m \mid \alpha \in \mathbb F\}$ forms a set of right coset representatives for the quotient $H^\prime/H_1^\prime$. 

For each $\alpha \in \mathbb F$, we define $f_\alpha: H \rightarrow \mathbb C$ with support $H_1^\prime (1 + \alpha \varpi^m)$ such that it takes the value $\psi^\prime(h_1)$ on $h_1(1 + \alpha \varpi^m)$ for each $h_1 \in H_1^\prime$. Note that $f_\alpha$'s form a basis for $W$. Observe that  
\begin{equation}\label{action1}
(1 + \alpha \varpi^m) \cdot f_\beta = \theta_\psi(1 + \alpha \varpi^m) f_\beta = \psi[1 +(\beta-\alpha)\alpha \varpi^{2m}] f_{\beta-\alpha},
\end{equation}
and 
\begin{equation}\label{action2}
(1 + \alpha \zeta \varpi^m) \cdot f_0 = \theta_\psi(1 + \alpha \zeta \varpi^m) f_0 = \psi^\prime(1 + \alpha \zeta \varpi^m) f_0.
\end{equation}

Now for each tuple $(i, \underline{\alpha}, \alpha_m)$ with $i \in \{0,1\}$, $\underline{\alpha} \in \mathcal{R}^\prime$,  and $\alpha_m \in \mathbb F$, we define two $W$-valued functions as follows. The function
$\phi_{i,\underline{\alpha},\alpha_m}: D^\times \rightarrow W$  is supported on $H \cdot \Delta_{i,\underline{\alpha}}$ with \[\phi_{i,\underline{\alpha},\alpha_m}(h \cdot \Delta_{i,\underline{\alpha}})= \theta_\psi^\#(h)(f_{\alpha_m})\]
and the function $\phi_{i,\underline{\alpha},\alpha_m}^\prime: D^\times \rightarrow W$ is supported on $H \cdot \Delta_{i,\underline{\alpha}}^\prime $ with 
\[\phi_{i,\underline{\alpha},\alpha_m}^\prime (h \cdot \Delta_{i,\underline{\alpha}}^\prime)= \theta_\psi^\#(h)(f_{\alpha_m})\]
for each $h \in H$.

In this case, we are going to exhibit a number of equivariant vectors for characters of $K^\times$ and $L^\times$. 

\subsection{Restricted to $K^\times$}

We first deal with $K = F(\varpi)$. 

\begin{proposition}\label{prop4}
Let $\chi$ be a smooth character of $K^\times$ with 
\[\chi|_{H_1^\prime \cap K^\times} =  \chi|_{F^\times (1+\mathcal P_K^{m+1})}  =  \psi^\prime|_{F^\times (1+\mathcal P_K^{m+1})},\]
when $m$ is odd, and 
\[\chi|_{H \cap K^\times} = \chi|_{F^\times (1 +\mathcal P_K^{m})} = \theta_\psi^\#|_{F^\times (1 + \mathcal P_K^{m})},\]
when $m$ is even. Then $\chi$ occurs with multiplicity one in the restriction of $\pi_\psi$ to $K^\times$. The vector 
\[v_\chi =
\begin{cases}
\sum\limits_{\substack{i=0,1 \\  \underline{\alpha} \in \mathcal{R}^\prime, \alpha_m \in \mathbb F }} \chi(\Delta_{i,\underline{\alpha}})~ \chi(1 + \alpha_m \varpi^m) \cdot  \phi_{i,\underline{\alpha}, \alpha_m}  &\text{if $m$ is odd,} \\\
\sum\limits_{\substack{i=0,1 \\  \underline{\alpha} \in \mathcal{R}^\prime  }} \chi(\Delta_{i,\underline{\alpha}}) \cdot  \phi_{i,\underline{\alpha}}    &\text{if $m$ is even,} 
\end{cases}\]
lies in the $\chi$-isotypic component of the representation $\pi_\psi$.
\end{proposition}

\begin{proof}
Suppose, $m$ is odd. Fix $x=h \cdot \Delta_{i_0, \underline{\delta}} \in D^\times$, where $h \in H$, $\underline{\delta}=(\delta_1, 0, \delta_3, 0, \dots, \delta_{m-2}, 0) \in \mathcal R$, and $i_0 \in \{0,1\}$. Note that 
\[v_{\chi}(x) = 
\begin{cases}
\sum\limits_{\alpha_m \in \mathbb F} \chi(\Delta_{i_0,\underline{\delta}}) \chi(1 + \alpha_m \varpi^m) \cdot \theta_\psi^\#(h)(f_{\alpha_m}) &\text{if $\underline{\delta} \in \mathcal R^\prime $,} \\
0 &\text{otherwise.}
\end{cases} \]
We now assume that $\underline{\delta} \in \mathcal R^\prime $. Let $z \in K^\times$ and write 
\[z= a \cdot \varpi^{j_0} \cdot (1 + \sum\limits_{\substack{r \geq 1  \\    r \text{ is odd} }}\beta_r  \varpi^r )\]
for some $a \in F^\times$, $j_0 \in \{0,1\}$, and $\beta_r \in \mathbb F$ for all $r$. Then,
\[ \pi_\psi(z) v_{\chi}(x) = v_{\chi}(xz) 
= v_{\chi}\left( ah \cdot \varpi^{i_0+j_0} (1 + \sum\limits_{\substack{r = 1  \\    r \text{ is odd} }}^{m-1} \delta_r  \varpi^r ) \cdot (1 + \sum\limits_{\substack{r \geq 1  \\    r \text{ is odd} }}\beta_r  \varpi^r )\right) .\]
We write  
\[\textbf{y}= \varpi^{i_0+j_0} (1 + \sum\limits_{\substack{r = 1  \\    r \text{ is odd} }}^{m-1} \delta_r  \varpi^r ) \cdot (1 + \sum\limits_{\substack{r \geq 1  \\    r \text{ is odd} }}\beta_r  \varpi^r )\]
as 
\[\textbf{y}= b_0 \cdot h_0 \cdot \Delta_{ i_0 + j_0, \underline{\delta}^\prime},\] 
where
\[b_0 = 1 + \sum\limits_{\substack{r = 2 \\    r \text{ is even} }}^{m-1} \xi_r  \varpi^r \]
with 
\[\xi_r=\sum\limits_{\substack{t=1  \\    t \text{ is odd} }}^{r-1} \delta_{r-t}\beta_t\]
for even $r \leq m-1$, and
\[\underline{\delta}^{\prime} = (\delta_1^{\prime},, 0, \delta_3^\prime, 0, \dots , \delta_{m-2}^{\prime},0) \in \mathcal{R}^{\prime}\]
with 
\[\delta_r^{\prime} = \delta_r + \beta_r  - \sum\limits_{\substack{t =1 \\    t \text{ is odd} }}^{r-2} \delta_t^{\prime} \xi_{r-t}\]
for odd $r \leq m-2$, and
\[h_0 = 1 + \sum\limits_{\substack{j =m  \\    j \text{ is odd} }}^{2m-1} \epsilon_j \varpi^j + \sum\limits_{\substack{k = m +1 \\    k \text{ is even} }}^{2m} \xi_k \varpi^k +  \cdots \]
such that
\[\epsilon_j = \beta_j -\sum\limits_{\substack{t=1  \\    t \text{ is odd} }}^{m-2} \delta_t^\prime \xi_{j-t}- \sum\limits_{\substack{t=m  \\    t \text{ is odd} }}^{j-2} \epsilon_t \xi_{j-t},\]
for odd $j\geq m$, and
\[\xi_k = \sum\limits_{\substack{t=1  \\    t \text{ is odd} }}^{m-2}\delta_t \beta_{k-t} -\sum\limits_{\substack{t=m+1  \\    t \text{ is even} }}^{k-2} \xi_t \xi_{k-t} - \sum\limits_{\substack{t=m  \\    t \text{ is odd} }}^{k-1} \delta_{k-t}^\prime \epsilon_{t}\]
for even $m+1 \leq k \leq 2m$. We take $\delta^\prime_m =0$ in the above. 

Clearly, $b_0 \in F^\times$, and $h_0 \in H \cap K^\times$. We now write 
\[h_0 = h_1 \cdot (1 + \epsilon_m \varpi^m),\]
where 
\[h_1=  1 + \sum\limits_{\substack{j =m+2  \\    j \text{ is odd} }}^{2m-1} \epsilon_j \varpi^j + \sum\limits_{\substack{k = m +1 \\    k \text{ is even} }}^{2m} \xi_k \varpi^k +  \cdots\]
belongs to $H_1^\prime \cap K^\times$.
We get,
\begin{align*}
\pi_\psi(z) v_{\chi}(x) &= v_{\chi} \left( ah \cdot b_0 \cdot h_0 \cdot \Delta_{i_0 + j_0, \underline{\delta}^\prime} \right)  \\
&= \theta_\psi^\#(ah \cdot b_0 \cdot h_0)  v_{\chi} \left( \Delta_{ i_0 + j_0, \underline{\delta}^\prime} \right) \\
&= \psi(ab_0) \chi( \Delta_{ i_0 + j_0, \underline{\delta}^\prime} )  \theta_\psi^\#(h) \theta_\psi(h_1)  \theta_\psi (1 + \epsilon_m \varpi^m) \left[ \sum\limits_{\alpha_m \in \mathbb F} \chi(1 + \alpha_m \varpi^m) f_{\alpha_m}\right]. 
\end{align*}
We now analyse the vector 
\begin{align*}
{\bf w} &= \theta_\psi(h_1)  \theta_\psi (1 + \epsilon_m \varpi^m) \left[ \sum\limits_{\alpha_m \in \mathbb F} \chi(1 + \alpha_m \varpi^m) f_{\alpha_m}\right] \\
&= \theta_\psi(h_1) \left[ \sum\limits_{\alpha_m \in \mathbb F} \chi(1 + \alpha_m \varpi^m) \cdot \psi \left(1 +(\alpha_m-\epsilon_m)\epsilon_m \varpi^{2m}\right) f_{\alpha_m-\epsilon_m} \right]  
\end{align*}		 
where in the last step we have made use of (\ref{action1}). Note that, by our assumption on $\chi$, we have $\theta_\psi(h_1) = \psi^\prime(h_1) =\chi(h_1)$ (since $h_1 \in H_1^\prime \cap K^\times$) and also $\chi$ is trivial on $1+\mathcal P_K^{2m+1}$ (since $\psi^\prime$ is trivial on $1+\mathcal P_D^{2m+1}$). Thus, we get
\begin{align*}
{\bf w} &= \chi (h_1)  \sum\limits_{\alpha_m \in \mathbb F} \chi(1 + \epsilon_m \varpi^m) \cdot \chi (1 +(\alpha_m-\epsilon_m) \varpi^m) f_{\alpha_m-\epsilon_m} \\
&= \chi(h_0)  \sum\limits_{\alpha_m \in \mathbb F} \chi(1+\alpha_m \varpi^m)f_{\alpha_m}
\end{align*}
Therefore,
\begin{align*}		
\pi_\psi(z) v_{\chi}(x) &=   \chi \left(a \cdot b_0 \cdot h_0 \cdot \Delta_{ i_0 + j_0, \underline{\delta}^\prime}  \right) 
\theta_\psi^\#(h) \left[ \sum\limits_{\alpha_m \in \mathbb F} \chi(1+\alpha_m \varpi^m)f_{\alpha_m} \right]\\
&= \chi \left(z \cdot \Delta_{ i_0, \underline{\delta}}  \right) \sum\limits_{\alpha_m \in \mathbb F} \chi(1 + \alpha_m \varpi^m) \cdot \theta_\psi^\#(h)(f_{\alpha_m})  \\
&= \chi(z) \cdot v_\chi(x).
\end{align*}

When $m$ is even, $W = \mathbb C$ and $\theta_\psi^\#$ is a one dimensional character. Now observe that the above proof in the odd case goes through, and in fact is simpler, with $\alpha_m=0$, $f_0 = 1$, and by taking $\phi_{i,\underline{\alpha}} = \phi_{i,\underline{\alpha},0}$.   
\end{proof}

\begin{remark}\label{rmk4}
The number of characters $\chi: K^\times \rightarrow \mathbb C$ satisfying $\chi |_{H_1^\prime \cap K^\times} = \psi^\prime|_{H_1^\prime \cap K^\times}$ is 
\[\left| \frac{K^\times}{F^\times (1 + \mathcal P_K^{m +1} )}   \right| = 2q^{m-\left[\frac{m}{2} \right]} = 2q^\frac{m+1}{2}, \] when $m$ is odd. Similarly, when $m$ is even, the number of characters $\chi: K^\times \rightarrow \mathbb C$ which satisfy $\chi|_{H \cap K^\times} = \theta_\psi^\# |_{H \cap K^\times}$ is 
\[\left| \frac{K^\times}{F^\times (1 + \mathcal P_K^m )}  \right| = 2q^{m-1-\left[\frac{m-1}{2} \right]} = 2q^\frac{m}{2}.\]
\end{remark}

If $\pi_\psi$ has trivial central character then the trivial character $\textbf{1}_{K^\times}$ of $K^\times$ occurs in the restriction $\pi_\psi$ to $K^\times$ and putting $\chi=1$ we get the following corollary.

\begin{corollary}\label{cor4}
 If $\pi_\psi$ has trivial central character then the vector 
 \[v_K = \begin{cases}
 \sum\limits_{\substack{i=0,1 \\  \underline{\alpha} \in \mathcal{R}^\prime, \alpha_m \in \mathbb F }}   \phi_{i,\underline{\alpha}, \alpha_m}  &\text{if $m$ is odd,}  \\
\sum\limits_{\substack{i=0,1 \\  \underline{\alpha} \in \mathcal{R}^\prime  }} \phi_{i,\underline{\alpha}}  &\text{if $m$ is even,} 
\end{cases}\]
is a $K^\times$-fixed vector in the representation $\pi_\psi$.
\end{corollary}

\subsection{Restricted to $L^\times$}

We take $L = F(\zeta \varpi)$. 

\begin{proposition}\label{prop5}
Let $\lambda$ be a smooth character of $L^\times$ with 
\[\lambda|_{H_1^\prime \cap L^\times} =  \lambda|_{F^\times (1+\mathcal P_L^m)}  =  \psi^\prime|_{F^\times (1+\mathcal P_L^m)},\]
when $m$ is odd, and 
\[\lambda|_{H \cap L^\times} = \lambda|_{F^\times (1 +\mathcal P_L^m)} =  \theta_\psi^\#|_{F^\times (1 + \mathcal P_L^m)},\]
when $m$ is even. Then $\lambda$ occurs with multiplicity one in the restriction of $\pi_\psi$ to $L^\times$. The vector 
\[v_\lambda = \begin{cases}
\sum\limits_{\substack{i=0,1 \\  \underline{\alpha} \in \mathcal{R}^\prime }} \lambda(\Delta_{i,\underline{\alpha}}^\prime) \cdot  \phi^\prime_{i,\underline{\alpha}, 0}  &\text{if $m$ is odd,} \\
\sum\limits_{\substack{i=0,1 \\  \underline{\alpha} \in \mathcal{R}^\prime  }} \lambda(\Delta_{i,\underline{\alpha}}^\prime) \cdot  \phi_{i,\underline{\alpha}}^\prime  &\text{if $m$ is even,}
\end{cases}\]
lies in the $\lambda$-isotypic component of the representation $\pi_\psi$.
\end{proposition}

\begin{proof}
The proof is similar to the proof of Proposition \ref{prop4}. Note that we take $\alpha_m = 0$ and instead of (\ref{action1}) we need to employ (\ref{action2}).	
\end{proof}	

\begin{remark}\label{rmk5}
 The number of characters $\lambda: L^\times \rightarrow \mathbb C$ satisfying $\lambda|_{H_1^\prime \cap L^\times} = \psi^\prime |_{H_1^\prime \cap L^\times}$ or $\lambda |_{H \cap L^\times} = \theta_\psi^\# |_{H \cap L^\times}$ is 
 \[\left| \frac{L^\times}{F^\times (1 + \mathcal P_L^{\left[\frac{m-1}{4} \right] +1} )}   \right| = 2q^{m-1-\left[\frac{m-1}{2} \right]}.\]
\end{remark}

If $\pi_\psi$ has trivial central character then the trivial character $\textbf{1}_{L^\times}$ of $L^\times$ occurs in the restriction $\pi_\psi$ to $L^\times$ and putting $\lambda =1$ we get the following corollary.

\begin{corollary}\label{cor5}
If $\pi_\psi$ has trivial central character then the vector 
\[v_L = \begin{cases}
\sum\limits_{\substack{i=0,1 \\  \underline{\alpha} \in \mathcal{R}^\prime}}   \phi_{i,\underline{\alpha},0}^\prime  &\text{if $m$ is odd,} \\
\sum\limits_{\substack{i=0,1 \\  \underline{\alpha} \in \mathcal{R}^\prime  }} \phi_{i,\underline{\alpha}}^\prime  &\text{if $m$ is even,} 
\end{cases}\]
is an $L^\times$-fixed vector in the representation $\pi_\psi$.
\end{corollary}

\section{Correlation Coefficients}\label{correlation}

In this section we prove the main theorem of this paper on correlation coefficients. As earlier, we have
\[M = F(\zeta), K = F(\varpi), L = F(\zeta\varpi).\] We take $G = D^\times/F^\times$ and consider given embeddings of $M^\times, K^\times$, and $L^\times$ in $G$. By abuse of notation, we denote the corresponding subgroups by $M, K$, and $L$, as this does not cause any confusion. We are interested in the correlation coefficients 
\[c(\pi;M,K), c(\pi;M,L), c(\pi;K,L)\]
for an irreducible representation $\pi$ of $G$.  

As mentioned in the introduction, our methods are straightforward. We work with a given model of $\pi=\pi_\psi$ and we have explicitly written down the invariant vectors for $M,K$, and $L$ in his model in Sections \ref{ramified} and \ref{unramified}. Now, in order to compute the correlation coefficients of the main theorem, we fix a $G$-invariant inner product on $\pi_\psi$, normalize the three invariant vectors to those of unit length in this inner product, and compute the inner product between any two of these vectors.   

The $G$-invariant inner product is indeed the obvious one. Let $G$ be a compact group with a closed subgroup $H$. Suppose $(\sigma, W)$ is an irreducible unitary representation of $H$ with respect to an inner product $\langle \cdot, \cdot \rangle_\sigma$. Then the induced representation ${\rm Ind}_H^G \sigma$ is unitary with respect to the inner product 
\[\langle \phi_1, \phi_2 \rangle = \int\limits_G \langle \phi_1(g) , \phi_2(g) \rangle_\sigma dg,\] 
where $dg$ is a Haar measure on $G$.

Recall that in Section \ref{ramified} and also when $m$ is even in Section \ref{ramified}, the inducing representation $\sigma$ is one dimensional, whereas when $m$ is odd in Section \ref{ramified}, the inducing representation $\sigma$ has dimension $q$.  

\subsection{Ramified case}

We are in the context of Section \ref{ramified}, and thus $\pi_\psi$ is an irreducible representation of $D^\times$ with odd conductor $m+2$, where $m$ is an odd positive integer. We further assume that it has trivial central character. Recall that then it always has a fixed vector for $M$ (cf. Corollary \ref{cor1}), it has a fixed vector for $K$ precisely when $q \equiv 3 \mod 4$ (cf. Corollary \ref{cor2}), and it has a fixed vector for $L$ precisely when $q \equiv 1 \mod 4$ (cf. Corollary \ref{cor3}). Observe that the corresponding unit vectors are:

\[\hat{v}_M =  \frac{1}{\sqrt{(q+1) q^{[\frac{m-1}{4}]}} } \sum\limits_{\substack{\underline{\alpha} \in S^\prime  \\    0 \leq i \leq q }} f_{\underline{\alpha}, i},\]
\[ \hat{v}_K = \frac{1}{\sqrt{2 q^{[\frac{m+1}{4}]}} }  \sum\limits_{\underline{\alpha} \in S^{\prime \prime} } \left[ f_{\underline{\alpha}, t_0} + \psi(\varpi) f_{\underline{\alpha}, 3t_0} \right], 
\hat{v}_L  =  \frac{1}{\sqrt{2 q^{[\frac{m+1}{4}]}} } \sum\limits_{\underline{\alpha} \in S^{\prime \prime} } \left[ f_{\underline{\alpha}, r_0} + \psi(\varpi) f_{\underline{\alpha}, 3r_0+1} \right].\]

\begin{proof}[Proof of Theorem \ref{thm-odd}]
First assume $q \equiv 3 \mod 4$ and in this case we only need to compute $c(\pi;M,K)$. This is given by
\begin{align*}
c(\pi;M,K)] &= |\left\langle \hat{v}_M, \hat{v}_K \right\rangle|^2\\
&= \frac{1}{2 (q+1) q^{\frac{m-1}{2}}} \cdot |\left\langle  f_{\underline{0}, t_0} + f_{\underline{0}, 3t_0}, f_{\underline{0}, t_0} + \psi(\varpi) f_{\underline{0}, 3t_0} \right\rangle|^2 \\
&=  \frac{|1 + \psi(\varpi)|^2}{2(q+1)q^{\frac{m-1}{2}}} \\ 
&= \frac{|1 + \psi(\varpi)|^2}{2 \dim \pi} \\
&=
\begin{cases}
\frac{2}{\dim \pi}  &\text{if $\psi(\varpi) = 1$,} \\
0  &\text{if $\psi(\varpi) = -1$.}
\end{cases}
\end{align*}

Now let $q \equiv 1 \mod  4$ and we need to compute $c(\pi;M,L)$. This is given by 
\begin{align*}
c(\pi; M, L) &= |\left\langle \hat{v}_M, \hat{v}_L \right\rangle|^2\\
&= \frac{1}{2 (q+1) q^{\frac{m-1}{2}}} \cdot |\left\langle  f_{\underline{0}, r_0} + f_{\underline{0}, 3r_0+1}, f_{\underline{0}, r_0} + \psi(\varpi) f_{\underline{0}, 3r_0+1} \right\rangle|^2 \\
&=  \frac{|1 + \psi(\varpi)|^2}{2(q+1)q^{\frac{m-1}{2}}} \\ 
&= \frac{|1 + \psi(\varpi)|^2}{2 \dim \pi} \\
&=
\begin{cases}
\frac{2}{\dim \pi}  &\text{if $\psi(\varpi) = 1$,} \\
0  &\text{if $\psi(\varpi) = -1$.}
\end{cases}
\end{align*}
\end{proof}

\begin{remark}\label{extra-proof}
Similarly, by making use of Remark \ref{extra-k} and Remark \ref{extra-l}, it is easy to compute the following correlation coefficients. For $q \equiv 3$ mod $4$, we have 
\[c(\pi; K_{s}, K_{s^\prime}) = 0\] for $s \neq s^\prime$, and  
\[c(\pi; M, K_{s})=  \frac{|1 + \psi(\varpi)|^2}{2 \dim \pi}.\]
For $q \equiv 1$ mod $4$, we have 
\[c(\pi; L_{s}, L_{s^\prime}) = 0\] for $s \neq s^\prime$, and  
\[c(\pi; M, L_{s})=  \frac{|1 + \psi(\varpi)|^2}{2 \dim \pi}.\]
\end{remark}

\subsection{Unramified case}

We are in the context of Section \ref{unramified}, and thus $\pi_\psi$ is an irreducible representation of $D^\times$ with even conductor $2m+2$. There are two cases to consider, namely (i) $m$ is even and (ii) $m$ is odd. We further assume that $\pi_\psi$ has trivial central character. Thus, $\pi_\psi$ has fixed vectors for both $K$ (cf. Corollary \ref{cor4}) and $L$ (cf. Corollary \ref{cor5}). Observe that the corresponding unit vectors are:
\begin{align*}
\hat{v}_K &= \frac{1}{\sqrt{2 q^{\frac{m}{2}}} } \sum\limits_{\substack{i=0,1 \\  \underline{\alpha} \in \mathcal{R}^\prime  }} \phi_{i,\underline{\alpha}}, \\
\hat{v}_L &= \frac{1}{\sqrt{2 q^{\frac{m}{2}}} } \sum\limits_{\substack{i=0,1 \\  \underline{\alpha} \in \mathcal{R}^\prime  }} \phi_{i,\underline{\alpha}}^\prime,
\end{align*}
when $m$ is even, and
\begin{align*}
\hat{v}_K &= \frac{1}{\sqrt{2 q^{\frac{m+1}{2}}} } \sum\limits_{\substack{i=0,1 \\  \underline{\alpha} \in \mathcal{R}^\prime ,\alpha_m \in \mathbb F }} \phi_{i,\underline{\alpha}, \alpha_m}, \\
\hat{v}_L &= \frac{1}{\sqrt{2 q^{\frac{m-1}{2}}} } \sum\limits_{\substack{i=0,1 \\  \underline{\alpha} \in \mathcal{R}^\prime }} \phi_{i,\underline{\alpha}, 0}^\prime,
\end{align*}
when $m$ is odd. 

\begin{proof}[Proof of Theorem \ref{thm-even}]
Noting that
\[\phi_{0,\underline{0}}^\prime= \phi_{0,\underline{0}} ~\&~ \phi_{1,\underline{0}}^\prime= \psi(\zeta^{-1}) \cdot \phi_{1,\underline{0}},\]
and
\[\phi_{0,\underline{0},\alpha}^\prime= \phi_{0,\underline{0},\alpha } ~\&~ \phi_{1,\underline{0},\alpha}^\prime= \psi(\zeta^{-1}) \cdot \phi_{1,\underline{0}, \alpha },\]
we get, when $m$ is even,
\begin{align*}
c(\pi; K, L) &= |\left\langle \hat{v}_K, \hat{v}_L \right\rangle|^2\\
&= \frac{1}{4 q^{m}} \cdot |\left\langle  \phi_{0,\underline{0}} + \phi_{1,\underline{0}}, \phi_{0,\underline{0}} + \psi(\zeta^{-1}) \phi_{1,\underline{0}} \right\rangle|^2 \\
&= \frac{|1 + \psi(\zeta)|^2}{4 q^{m}}\\
&= \frac{|1 + \psi(\zeta)|^2}{2 \dim \pi},
\end{align*}
and when $m$ is odd,
\begin{align*}
c(\pi; K, L) &= |\left\langle \hat{v}_K, \hat{v}_L \right\rangle|^2 \\
&= \frac{1}{4 q^{m}} \cdot |\left\langle  \sum\limits_{\alpha \in \mathbb F} \phi_{0,\underline{0}, \alpha} + \sum\limits_{\alpha \in \mathbb F} \phi_{1,\underline{0}, \alpha}, \phi_{0,\underline{0}, 0} + \psi(\zeta^{-1})\phi_{1,\underline{0},0} \right\rangle |^2 \\
&= \frac{|1 + \psi(\zeta)|^2}{4 q^{m}}\\
&= \frac{|1 + \psi(\zeta)|^2}{2 \dim \pi}. 
\end{align*}
\end{proof}

\subsection{Compatibility with Schur orthogonality relations}\label{schur}

In this section, we verify that Theorem \ref{thm-odd} and Theorem \ref{thm-even} agree with the identity \cite[Equation (1) on p. 502]{aj22} 
\[\sum_{\pi \in \widehat{G}} \dim \pi \cdot c(\pi;H,K) = \frac{|G||H \cap K|}{|H||K|},\]
which is valid for any finite group $G$ with two multiplicity free subgroups $H$ and $K$. This above identity follows easily from \cite[Lemma 2.1]{aj22} by Schur orthogonality relations. 

The idea, as mentioned in Section \ref{intro}, is to work with the finite group $G_c = \frac{D^\times}{F^\times(1+\mathcal P_D^{c-1})}$, and the corresponding embeddings $M_c, K_c$, and $L_c$ of tori in $G_c$. Note that any two of these subgroups intersect trivially. The cardinalities of these groups are given by
\[|G_c| = \begin{cases}
2(q+1)q^{3j-1} &\text{if $c=2j+1$,} \\
2(q+1)q^{3j} &\text{if $c=2j+2$},
\end{cases}\] and
\[|M_c| = (q+1)q^{\left[\frac{c-2}{2} \right]}, |K_c|=|L_c| = 2q^{(c-2)-\left[\frac{c-2}{2}\right]}.\]

We need to know the number of irreducible representations of $D^\times$ of a given conductor with trivial central character and this is given by \cite[Proposition 3.5 and Remark on p. 189]{tun78}:
\[\mbox{Number of~} \pi \mbox{~of conductor~} c = \begin{cases}
\frac{q^2-1}{2}q^{\frac{c-4}{2}} &\text{if $c$ is even and $c \geq 4$,} \\
2(q-1)q^{\frac{c-3}{2}} &\text{if $c$ is odd and $c \geq3$.}
\end{cases}\]
Note that the only representation of conductor $1$ is the trivial representation and the number of irreducible representations of conductor $2$ is the number of two dimensional representations of the dihedral group $D_{2(q+1)}$ which is $\frac{q-1}{2}$. 

We also note that many correlation coefficients vanish because of lack of an invariant vector. More precisely, when $\pi$ has odd conductor, $c(\pi; M, K) = 0$ when $q \equiv 1 \mod 4$, $c(\pi; M, L) = 0$ when $q \equiv 3 \mod 4$,  and $c(\pi; K, L) = 0$ always (cf.  Remark \ref{trivial}). When $\pi$ has even conductor, $c(\pi; M, K) = c(\pi; M, L) = 0$ (cf. Remark \ref{trivial1}).

We first verify the identity involving $c(\pi; M, K)$ and the verification for the one involving $c(\pi;M, L)$ being similar. Only representations with odd conductor, say $c = 2j+1$, contribute to the sum on the left hand side. Half the representations of a given odd conductor arise from $K$ and the other half from $L$. Those representations arising from $L$ do not contribute if $q \equiv 3 \mod 4$ and those arising from $K$ do not contribute if $q \equiv 1 \mod 4$.  And amongst the representations $\pi_\psi$ that can potentially contribute to the sum, precisely half of these contribute non-trivially, namely $\frac{2}{\dim \pi}$, and the other half do not contribute, since half the time $\psi(\varpi) = 1$ and half the time $\psi(\varpi) = -1$. Thus,
\begin{align*}
\sum_{\pi \in \widehat{G}} \dim \pi \cdot c(\pi; M, K) &= 1 + \sum_{i=3}^c \frac{2(q-1)q^{\frac{i-3}{2}}}{4} \cdot \dim \pi \cdot \frac{2}{\dim \pi} \\
&= 1+ (q-1) \left(1+q+\dots+q^{\frac{c-3}{2}} \right) \\
&= q^{\frac{c-1}{2}} = q^j = \frac{2(q+1)q^{3j-1}}{(q+1)q^{j-1} \cdot 2q^j} \\
&= \frac{|G_c|}{|M_c||K_c|}.
\end{align*}

Next we verify the identity involving $c(\pi; K, L)$. This time, only representations with even conductor, say $c = 2j+2$, will matter. Note that
\[\sum_\psi |1+\psi(\zeta)|^2 = \sum_\psi (2+\psi(\zeta)+\psi(\zeta^{-1})) = 
\begin{cases}
q-1 &\text{if $i=0$,} \\
(q^2-1)q^{i-1} &\text{if $i \geq 1$,} 
\end{cases}
\]
where the sum is over minimal characters of $M^\times$ of conductor $i+1$. We have
\begin{align*}
\sum_{\pi \in \widehat{G}} \dim \pi \cdot c(\pi; K, L) &= 1 + \frac{q-1}{2} + \sum_{i=1}^j \frac{q^2-1}{2}q^{i-1} \\
&= \frac{q+1}{2} q^j = \frac{2(q+1)q^{3j}}{2q^j \cdot 2q^j} \\
&=  \frac{|G_c|}{|K_c||L_c|}.
\end{align*}

\subsection{Proofs of Theorem \ref{thm-main} and Theorem \ref{thm-main-1}}\label{proof-main}

Let $\pi$ be an irreducible representation of $D^\times$ which has fixed vectors for some tori, it follows that $\omega_\pi = 1$. In particular, 
\[\varepsilon(\pi,\Psi_a) = \omega_\pi(a) \varepsilon(\pi,\Psi) = \varepsilon(\pi,\Psi)\]
is independent of the the additive character $\Psi$ of $F$, so we denote the epsilon factor by $\varepsilon(\pi)$. We have $\pi \equiv \pi^\vee$, since $\omega_\pi = 1$.  It now follows from the identity
\[\varepsilon(\pi)\varepsilon(\pi^\vee) = \omega_\pi(-1)\] that $\varepsilon(\pi) \in \{\pm 1\}$. 

Suppose $\pi = \pi_\psi$ is constructed from a minimal (primordial in the terminology of \cite{gk80}) character $\psi$ of $\Omega^\times$ as in Section \ref{prelim}, where $\Omega$ is a quadratic extension of $F$ . Then, by \cite[Theorem 3.7 and Proposition 4.5]{gk80}, we have
\[\varepsilon(\pi_\psi,\Psi) = (-1)^{f(\psi)}\varepsilon(\psi, \Psi \circ {\rm Tr}_{\Omega/F}).\] 
Note that $\psi|_{F^\times} = \omega_\pi = 1$, and therefore by \cite[Theorem 3]{fq73}, we get
\[\varepsilon(\psi,\Psi \circ {\rm Tr}_{\Omega/F}) = \psi(\Delta),\]
where $\Delta$ is any element of $\Omega^\times$ such that Tr$_{\Omega/F}(\Delta) = 0$.

Note that we may take $\Omega = K$ and $\Delta = \varpi$ when $\pi$ has odd minimal conductor. Note also that $f(\psi)$ is even in this case. It follows that
\[\varepsilon(\pi) = \psi(\varpi).\]
 It is easy to see either from the character formulae for irreducible representations of $D^\times$ or from Tunnell's result \cite{tun83,sai93} that such a $\pi$ always has a non-trivial $M$-fixed vector and it has a non-trivial $K$-fixed (resp. $L$-fixed) vector when $q \equiv 3 \mod 4$ (resp. $q \equiv 1 \mod 4$). Now Theorem \ref{thm-main} follows from Theorem \ref{thm-odd}.

Similarly, Theorem \ref{thm-main-1} follows from Theorem \ref{thm-even} and Remark \ref{rmk-iso-1} by noting that $\Delta = \zeta^{\frac{q+1}{2}}$ when $\Omega = M$ and in this case we have $f(\psi) = f(\pi)/2$.  

\section*{Acknowledgements}

The authors are grateful to Dipendra Prasad for suggesting several ideas, many fruitful discussions, and for his encouragement for this work. The second author thanks the Institute of Eminence Cell at IIT Bombay for supporting his postdoctoral position with an Institute Post-doctoral Fellowship.


\end{document}